\documentclass[a4paper]{amsart}       

\usepackage{graphicx}
\usepackage[all]{xy}
\usepackage{color}
\usepackage{hyperref}
\usepackage{enumerate}
\usepackage{amsmath}
\usepackage{amsthm}
\usepackage{amssymb}
\usepackage{amscd}
\usepackage{epsfig}
\usepackage[english]{babel}
\usepackage[stable]{footmisc}
\usepackage{ucs}
\usepackage[utf8x]{inputenc}

\newcommand{\cG}{\mathcal{G}}

\newcommand{\G}{\mathcal{G}}            

\numberwithin{equation}{section}
\allowdisplaybreaks

\newtheorem{theorem}{Theorem}[section]
\newtheorem{lemma}[theorem]{Lemma}
\newtheorem{proposition}[theorem]{Proposition}
\newtheorem{corollary}[theorem]{Corollary}

\theoremstyle{definition}
\newtheorem{definition}[theorem]{Definition}

\newtheorem{remark}[theorem]{Remark}

\begin{document}

\title{The transverse density bundle and modular classes of Lie groupoids}

\author{Marius Crainic}
\address{Mathematical Institute, Utrecht University,
The Netherlands}
\email{m.crainic@uu.nl}

\author{Jo\~ao Nuno Mestre}
\address{Centre for Mathematics, University of Porto, Portugal}
\email{jnmestre@gmail.com}

\begin{abstract} In this note we revisit the notions of transverse density bundle and of modular classes of Lie algebroids and Lie groupoids; in particular, we point out that one should use the transverse density bundle $\mathcal{D}_{A}^{\textrm{tr}}$ instead of $Q_A$, which is the representation that is commonly used when talking about modular classes. One of the reasons for this is that, as we will see, $Q_A$ is not really an object associated to the stack presented by a Lie groupoid (in general, it is not a representation of the groupoid!). 
\end{abstract}

\maketitle


\section{Introduction}\label{intro}

We revisit the transverse density bundle $\mathcal{D}_{A}^{\textrm{tr}}$, a representation that is canonically associated to any Lie groupoid \cite{measures_stacks}, and its relation to the notions of modular classes of Lie groupoids and Lie algebroids. These are characteristic classes of certain 1-dimensional representations \cite{VanEst, QA, mehta-modular}, that we recall and discuss in this note.

The definition of the modular class of a Lie algebroid always comes with the slogan, inspired by various examples, that it is ``the obstruction to the existence of a transverse measure''. Here we would like to point out that the transverse density bundle $\mathcal{D}_{A}^{\textrm{tr}}$ and our discussions make this slogan precise. In particular, we point out that the canonical representation $Q_A$  \cite{QA} which is commonly used in the context of modular classes of Lie algebroids should actually be replaced by $\mathcal{D}_{A}^{\textrm{tr}}$, especially when passing from Lie algebroids to Lie groupoids.

\ \

\paragraph{\textbf{Notation and conventions:}}

Throughout the paper $\cG$ denotes a Lie groupoid over $M$, with source and target maps denoted by $s$ and $t$, respectively. For background material on Lie groupoids, Lie algebroids, and their representations we refer to \cite{ieke_mrcun}.   
We only consider Lie groupoids having all $s$-fibers of the same dimension or, equivalently, whose algebroid $A$ has constant rank.
Actually, all vector bundles in this paper are assumed to be of constant rank.

\section{Transverse density bundles}

\subsection{Volume, orientation and density bundles}\label{sub-dens-bundle}

First of all, note that any group homomorphism $\delta: GL_r\rightarrow \mathbb{R}^*$ allows us to associate to any $r$-dimensional vector space a canonical 1-dimensional vector space
\[ L_{\delta}(V):= \{\xi: \textrm{Fr}(V)\rightarrow \mathbb{R}\ :\  \xi(e\cdot A)= \delta(A) \xi(e) \ \textrm{for\ all} \ e\in \textrm{Fr}(V),\ A\in GL_r\},\]
where $\textrm{Fr}(V)= \textrm{Isom}(\mathbb{R}^r, V)$ is the space of frames on $V$, endowed with the standard (right) action of $GL_r$. The cases that are of interest for us are:
\begin{itemize}
\item $\delta= \textrm{det}$, for which we obtain the top exterior power $\Lambda^{\textrm{top}}V^*$.
\item $\delta= \textrm{sign}\circ \textrm{det}$, for which we obtain the \textbf{orientation space} $\mathfrak{o}_{V}$ of $V$.
\item $\delta= |\textrm{det}|$ which defines the \textbf{space $\mathcal{D}_{V}$ of densities} of $V$. More generally, for $l\in \mathbb{Z}$, $\delta= |\textrm{det}|^l$ defines the \textbf{space $\mathcal{D}_{V}^{l}$ of $l$-densities} of $V$.
\end{itemize}
It is clear that there is a canonical isomorphism:
\[ \mathcal{D}_V\otimes \mathfrak{o}_{V}\cong \Lambda^\mathrm{top} V^*,.\]
When $V= L$ is 1-dimensional, $\mathcal{D}_{L}$ is also denoted by $|L|$; so, in general,
\[ \mathcal{D}_{V}= |\Lambda^\mathrm{top} V^*|,\]
which fits well with the fact that for any $\omega\in \Lambda^\mathrm{top} V^*$, $|\omega|$ makes sense as an element of $\mathcal{D}_{V}$.
For 1-dimensional vector spaces $W_1$ and $W_2$, one has a canonical isomorphism $|W_1| \otimes |W_2| \cong |W_1\otimes W_2|$
(in particular $|W^*|\cong |W|^*$). From the properties of $\Lambda^\mathrm{top}V^*$ (or by similar arguments), one obtains canonical isomorphisms:
\begin{enumerate}
\item[1.] $\mathcal{D}_{V}^{*}\cong \mathcal{D}_{V^*}$ for any vector space $V$.
\item[2.] For any short exact sequence of vector spaces
\[ 0\rightarrow V\rightarrow U\rightarrow W\rightarrow 0\]
(e.g. for $U= V\oplus W$), one has an induced isomorphism $\mathcal{D}_{U}\cong \mathcal{D}_{V}\otimes \mathcal{D}_{W}$.
\end{enumerate}

Since the previous discussion is canonical (free of choices), it can applied (fiberwise) to vector bundles over a manifold $M$ so that, for any such vector bundle $E$, one can talk about the associated line bundles over $M$
\[ \mathcal{D}_E,\ \Lambda^\mathrm{top} E^*,\ \mathfrak{o}_{E}\]
and the previous isomorphisms continue to hold at this level. However, at this stage, only $\mathcal{D}_E$ is trivializable, and even that is in a non-canonical way.

\begin{definition} A \textbf{density on a manifold} $M$ is any section of the density bundle $\mathcal{D}_{TM}$. 
\end{definition}

The main point about densities on manifolds is that they can be integrated in a canonical fashion, so that associated to any compactly supported positive density $\rho$ on $M$, one obtains a Radon measure $\mu_{\rho}$  defined by \[ \mu_{\rho}: C_{c}^{\infty}(M)\rightarrow \mathbb{R},\ \ \ \ \mu_{\rho}(f)= \int_{M} f\cdot \rho .\]

\subsection{The transverse volume, orientation and density bundles}\label{sub-tr-dens-bundle}

\begin{definition}\label{def-tr-gpd-dens} For a Lie algebroid $A$ over $M$, the \textbf{the transverse density bundle of $A$} is the vector bundle over $M$ defined by: 
\[ \mathcal{D}_{A}^{\textrm{tr}}:= \mathcal{D}_{A^*}\otimes \mathcal{D}_{TM} .\]
\end{definition}

Similarly one can define the \textbf{transverse volume and orientation bundles}
\[ \mathcal{V}^{\textrm{tr}}_{A}:= \mathcal{V}_{A^*}\otimes \mathcal{V}_{TM}= \Lambda^{\textrm{top}}A\otimes \Lambda^{\textrm{top}}T^*M,\ \ \ \mathfrak{o}^{\textrm{tr}}_{A}:= \mathfrak{o}_{A^*}\otimes \mathfrak{o}_{TM}\]
and the usual relations between these bundles continue to hold in this setting; e.g.:
\[ \mathcal{D}_{A}^{\textrm{tr}}= |\Lambda^{\textrm{top}}A\otimes \Lambda^{\textrm{top}}T^*M|= |\mathcal{V}^{\textrm{tr}}_{A}| . \]

One of the main properties of these bundles is that they are representations of $A$ and, even better, of $\G$, whenever $\G$ is a Lie groupoid with algebroid $A$; hence they do deserve the name of ``transverse'' vector bundles. We describe the canonical action of $\G$ on the transverse density bundle $\mathcal{D}_{A}^{\textrm{tr}}$; for the other two the description is identical. We have to associate to any arrow $g: x\rightarrow y$ of $\G$ a linear transformation 
\[ g_{*}: \mathcal{D}_{A, x}^{\textrm{tr}}\rightarrow \mathcal{D}_{A, y}^{\textrm{tr}}.\]
The differential of $s$ and the right translations induce a short exact sequence

\[ 0 \rightarrow A_y \rightarrow T_g\G \stackrel{ds}{\rightarrow} T_{x}M \rightarrow 0 ,\]
which, in turn (cf. item 2 in Subsection \ref{sub-dens-bundle}), induces an isomorphism:  
\begin{equation}\label{DTG-dec} 
\mathcal{D}(T_g\G)\cong \mathcal{D}(A_y)\otimes \mathcal{D}(T_xM).
\end{equation}
Using the similar isomorphism at $g^{-1}$ and the fact that the differential of the inversion map gives an isomorphism $T_g\G\cong T_{g^{-1}}\G$, we find an isomorphism
\[ \mathcal{D}(A_y)\otimes \mathcal{D}(T_xM)\cong \mathcal{D}(A_x)\otimes \mathcal{D}(T_yM).\]
and therefore an isomorphism:
\[ \mathcal{D}(A_{x}^{*})\otimes \mathcal{D}(T_xM)\cong \mathcal{D}(A_{y}^{*})\otimes \mathcal{D}(T_yM),\]
and this defines the action $g_*$ we were looking for (it is straightforward to check that this defines indeed an action). 

\begin{definition}  A \textbf{transverse density} for the Lie groupoid $\G$ is any $\G$-invariant section of the transverse density bundle 
$\mathcal{D}_{A}^{\textrm{tr}}$.
\end{definition}

\begin{remark}Recall that the canonical integration of densities lets us associate to a compactly supported positive density $\rho$ on a manifold $M$ a Radon measure $\mu_\rho$. Similarly, a positive transverse density for a groupoid $\G$ gives rise to what is called a transverse measure for $\G$. Such measures were studied in \cite{measures_stacks}, and represent measures on the differentiable stack presented by $\G$.
\end{remark}

\section{The modular class(es) revisited}\label{sec-The modular class(es) revisited}

Throughout this section $\G$ is a Lie groupoid over $M$ and $A$ is its Lie algebroid. We will be using the transverse density bundle $\mathcal{D}_{A}^{\textrm{tr}}$, volume bundle $\mathcal{V}_{A}^{\textrm{tr}}$ and orientation bundle $\mathfrak{o}_{A}^{\textrm{tr}}$, viewed as representations of $\G$ as explained in Section \ref{sub-tr-dens-bundle}. Let us mention, right away, the relation between these bundles. As vector bundles over $M$, we know (see Section \ref{sub-dens-bundle}) that there are canonical vector bundle isomorphism between
\begin{itemize}
\item $\mathcal{D}_{A}^{\textrm{tr}}$ and $\mathcal{V}_{A}^{\textrm{tr}}\otimes \mathfrak{o}_{A}^{\textrm{tr}}$.
\item $\mathcal{V}_{A}^{\textrm{tr}}$ and $\mathcal{D}_{A}^{\textrm{tr}}\otimes \mathfrak{o}_{A}^{\textrm{tr}}$.
\item $\mathfrak{o}_{A}^{\textrm{tr}}\otimes \mathfrak{o}_{A}^{\textrm{tr}}$ and the trivial line bundle.
\item $\mathfrak{o}_{A}^{\textrm{tr}}$ and $(\mathfrak{o}_{A}^{\textrm{tr}})^*$.
\end{itemize}

\begin{lemma}\label{lemma-can-izos} All these canonical isomorphisms are isomorphisms of representations of $\G$ (where the trivial line bundle is endowed with the trivial action).  
\end{lemma}

\begin{proof} Given the way that the action of $\G$ was defined (Section \ref{sub-tr-dens-bundle}), the direct check can be rather lengthy and painful. Here is a more conceptual approach. The main remark is that these actions can be defined in general, whenever we have a functor 
$F$ which associates to a vector space $V$ a $1$-dimensional vector space $F(V)$ and to a (linear) isomorphism $f: V\rightarrow W$ an isomorphism $F(f): F(V)\rightarrow F(W)$ such that:
\begin{enumerate}
\item[1.] $F$ commutes with the duality functor $D$, i.e., $F\circ D$ and $D\circ F$ are isomorphic through a natural transformation $\eta: F\circ D \rightarrow D\circ F$. 
\item[2.] for any exact sequence $0\rightarrow U\rightarrow V\rightarrow W\rightarrow 0$ there is an induced isomorphism between $F(V)$ and $F(U)\otimes F(W)$, natural in the obvious sense.
\end{enumerate}
Let's call such $F$'s ``good functors''. The construction from Section \ref{sub-tr-dens-bundle} shows that for any good functor $F$, 
\[ F_{A}^{\textrm{tr}}:= F(A^*)\otimes F(TM) \]
is a representation of $\G$. Given two good functors $F$ and $F'$, an isomorphism $\eta: F\rightarrow F'$ will be called good if it is compatible with the natural transformations from 1. and 2. above.  It is clear that, for any such $\eta$, there is an induced map 
$\eta^{\textrm{tr}}$ that is an isomorphism between $F_{A}^{\textrm{tr}}$ and $F_{A}^{'\, \textrm{tr}}$, as representations of $\G$. It should also be clear that, for any two good functors $F$ and $F'$, so is their tensor product. We see that we are left with proving that certain isomorphisms involving the functors $\mathcal{D}$, $\mathcal{V}$ and $\mathfrak{o}$ (e.g. $\mathcal{D}\cong \mathcal{V}\otimes \mathfrak{o}$) are good in the previous sense; and that is straightforward. 
\end{proof}

\subsection{The modular class of $\G$}\label{subsec-The modular class of G}

Let us concentrate on the question of whether $\G$ admits a strictly positive transverse density (these are the ``measures'' from the slogan at the start of the section, or ``geometric measures'' in the terminology of \cite{measures_stacks}). Start with any strictly positive section $\sigma$ of $\mathcal{D}_{A}^{\textrm{tr}}$. Then any other such section is of type $e^f \sigma$ for some $f\in C^{\infty}(M)$; moreover $e^f \sigma$ is invariant if and only if 
\[ e^{f(y)} \sigma(y)= e^{f(x)} g(\sigma(x))\]
for all $g: x\rightarrow y$ an arrow of $\G$. Considering 
\[ c_{\sigma}(g):= ln\left(\frac{\sigma(y)}{g(\sigma(x))}\right),\]
one has $c_{\sigma}\in C^{\infty}(\G)$ and one checks right away that it is a 1-cocycle, i.e., 
\[ c_{\sigma}(gh)= c_{\sigma}(g)+ c_{\sigma}(h) \]
for all $g$ and $h$ composable. The condition on $f$ that we were considering reads:
\[ f(x)- f(y)= c_{\sigma}(g)\]
for all $g:x\rightarrow y$, i.e., $c_{\sigma}= \delta(f)$ in the differentiable cohomology complex of $\G$,  $(C^{\bullet}_{\textrm{diff}}(\G), \delta)$. Furthermore, an easy check shows that the class $[c_{\sigma}]\in H^{1}_{\textrm{diff}}(\G)$ does not depend on the choice of $\sigma$. Therefore it gives rise to a canonical class
\[ \textrm{mod}(\G)\in H^{1}_{\textrm{diff}}(\G),\]
called \textbf{the modular class of the Lie groupoid $\G$}. By construction:

\begin{lemma}\label{mod-as-obstr} $\G$ admits a strictly positive transverse density iff $\textrm{mod}(\G)= 0$.
\end{lemma}

The result makes precise the expectations of \cite[Section 7]{mehta-modular} for the meaning of the modular class in the absence of superorientability. With this, the existence of transverse densities and measures for proper Lie groupoids of \cite{measures_stacks} is just about the vanishing of differentiable cohomology of proper groupoids (Proposition $1$ in \cite{VanEst}).

The construction of $\textrm{mod}(\G)$ can be seen as a very particular case of the construction from \cite{VanEst} of characteristic classes of representations of $\G$, classes that live in the odd differentiable cohomology of $\G$. 

Here we are interested only in the $1$-dimensional representations $L$, with corresponding class denoted

\[ \theta_{\G}(L)\in H^{1}_{\textrm{diff}}(\G).\]
For a direct description, similar to that of $\textrm{mod}(\G)$, we first assume that $L$ is trivializable as a vector bundle and we choose a nowhere vanishing section $\sigma$. Then, for $g: x\rightarrow y$, we can write 
\[ g\cdot \sigma(x)= \tilde{c}_{\sigma}(g)  \sigma(x)\ \ \ \ (\tilde{c}_{\sigma}(g)\in \mathbb{R}^*)\]
and this defines a function 
 \begin{equation}\label{tilde-c-sigma} 
 \tilde{c}_{\sigma}: \G \rightarrow \mathbb{R}^* 
 \end{equation}
that is a groupoid homomorphism. The cocycle of interest is 
 \begin{equation}\label{tilde-c-sigma-form} 
c_{\sigma}=\textrm{ln}(|\tilde{c}_{\sigma}|): \G \rightarrow \mathbb{R};
\end{equation}
its cohomology class does not depend on the choice of $\sigma$ and defines $\theta_{\G}(L)$.  It is clear that for two such representations $L_1$ and $L_2$ (trivializable as vector bundles),
\begin{equation}\label{multiplicativiti-theta} 
\theta_{\G}(L_1\otimes L_2)= \theta_{\G}(L_1)+ \theta_{\G}(L_2).
\end{equation}
This indicates how to proceed for a general $L$: consider the representation $L\otimes L$ which is (noncanonically) trivializable and define:
\begin{equation}\label{multiplicativiti-theta-trick}  
\theta_{\G}(L):= \frac{1}{2} \theta_{\G}(L\otimes L).
\end{equation}
The multiplicativity formula for $\theta_{\G}$ remains valid for all $L_1$ and $L_2$. By construction:

\begin{lemma} One has $\textrm{mod}(\G)= \theta_{\G}(\mathcal{D}_{A}^{\textrm{tr}})$. 
\end{lemma}

\begin{remark}[a warning] \label{remark-warning} It is not true (even if $L$ is trivializable as a vector bundle!) that $\theta_{\G}(L)$ is the obstruction to $L$ being isomorphic to the trivial representation. Lemma \ref{mod-as-obstr} holds because the transverse density bundle is more than trivializable: one can also talk about positivity of sections of $\mathcal{D}_{A}^{\textrm{tr}}$ and $\mathcal{D}_{A}^{\textrm{tr}}$ is trivializable as an {\it oriented} representation of $\G$.
\end{remark}

The tendency in existing literature, at least for the infinitesimal version of the modular class (see below), is to use simpler representations instead of $\mathcal{D}_{A}^{\textrm{tr}}$. Here we would like to clarify the role of the transverse volume bundle $\mathcal{V}_{A}^{\textrm{tr}}$: can one use it to define $\textrm{mod}(\G)$? In short, the answer is: yes, but one should not do it because it would give rise to the wrong expectations (because of the previous warning!). We summarise this into the following:

\begin{proposition}\label{not-col-tr} For any Lie groupoid $\G$, $\textrm{mod}(\G)= \theta_{\G}(\mathcal{V}_{A}^{\textrm{tr}})$. However, it is not true that 
that $\textrm{mod}(\G)= 0$ happens if and only if $\G$ admits a transverse volume form (i.e., a nowhere vanishing $\G$-invariant  section of $\mathcal{V}_{A}^{\textrm{tr}}$).
\end{proposition}

Counterexamples for the last part are provided already by manifolds $M$, viewed as groupoids with unit arrows only.
Indeed, in this case the associated transverse (density, volume) bundles are the usual bundles of $M$; hence the modular class is zero even if $M$ is not orientable. For the first part of the proposition, using the multiplicativity (\ref{multiplicativiti-theta}) of $\theta_{\G}$ and the canonical isomorphisms discussed at the beginning of the section, we have to show that
\begin{equation}\label{vanish}
\theta_{\G}(\mathfrak{o}_{A}^{\textrm{tr}})= 0.
\end{equation}
In turn, this follows by applying again the multiplicativity of $\theta_{\G}$ and the canonical isomorphism between  $\mathfrak{o}_{A}^{\textrm{tr}}\otimes \mathfrak{o}_{A}^{\textrm{tr}}$ and the trivial representation.

\subsection{The modular class of $A$}\label{subsec-The modular class of A} 
The construction of the modular class of a Lie algebroid $A$, introduced by Evens, Lu and Weinstein \cite{QA}, 
 is based on the geometry of a certain $1$ - dimensional representation $Q_A$ of the Lie algebroid $A$: $\textrm{mod}(A)$ is the characteristic class of $Q_A$. 
 Let us first recall the construction of the characteristic class $\theta_{A}(L)\in H^1(A)$ associated to any $1$-dimensional representation $L$ of $A$ (the infinitesimal version of the construction of the classes $\theta_{\G}(L)$ of groupoid representations). First one uses the analogue of (\ref{multiplicativiti-theta-trick}) to reduce the construction to the case when $L$  is trivializable as a vector bundle; then, for such $L$, one chooses a nowhere vanishing section $\sigma$ and one writes the infinitesimal action $\nabla$ of $A$ on $L$ as
\[ \nabla_{\alpha}(\sigma)= c_{\sigma}(\alpha) \cdot \sigma,\]
therefore defining $c_{\sigma}$ as an element $c_{\sigma}(L)\in \Omega^1(A)$. 
Similar to the previous discussion, the flatness of $\nabla$ implies that $c_{\sigma}(L)$ is a closed $A$-form and its cohomology class does not depend on the choice of $\sigma$; therefore it defines a class, called the characteristic class of $L$, and denoted
\[ \theta_A(L)\in H^1(A).\]
Note that the situation is simpler than at the level of $\G$: for $L$ trivializable as a vector bundle, $\theta_A(L)= 0$ if and only if $L$ is isomorphic to the trivial representation of $A$ (compare with the warning from Remark \ref{remark-warning}!). 

Inspired by the previous subsection, we define:

\begin{definition}
The \textbf{modular class of a Lie algebroid} $A$, denoted 
$\textrm{mod}(A)$, is the characteristic class of $\mathcal{D}_{A}^{\textrm{tr}}$.
\end{definition}

When $A$ is the Lie algebroid of a Lie groupoid $\cG$, since 
$\mathcal{D}_{A}^{\textrm{tr}}$ is a representation of $\cG$, we deduce (cf. Theorem 7 in \cite{VanEst}):

\begin{proposition}\label{mod-cor1} For any Lie groupoid $\G$, the Van Est map in degree $1$,
\[ VE: H^{1}_{\mathrm{diff}}(\G) \rightarrow H^1(A)\]
sends $\textrm{mod}(\G)$ to $\textrm{mod}(A)$. In particular, if $A$ is integrable by a unimodular Lie groupoid (e.g. by a proper Lie groupoid), then its modular class vanishes.
\end{proposition}

In particular, since the Van Est map in degree 1 is injective if $\cG$ is $s$-connected (see e.g. Theorem 4 in \cite{VanEst}) we deduce:

\begin{corollary}\label{mod-cor2} If $\cG$ is an $s$-connected Lie group with Lie algebroid $A$, then $\textrm{mod}(A)$ is the obstruction to the existence of a strictly positive transverse density of $\G$.
\end{corollary}

\subsection{(Not) $Q_A$}\label{subsec-not QA} 
The modular class of a Lie algebroid $A$ can be defined as the characteristic class of various $1$-dimensional representations of $A$. We have used $\mathcal{D}_{A}^{\textrm{tr}}$, but the common choice in the literature (starting with \cite{QA}) is the line bundle 
\[ Q_{A}= \Lambda^{\textrm{top}}A\otimes |\Lambda^{\textrm{top}}T^*M| .\]
The infinitesimal action of $A$ on $Q_A$ is explained in \cite{QA}; equivalently, one writes 
\begin{equation}\label{dec-QA} 
Q_A= \mathcal{D}_{A}^{\textrm{tr}}\otimes \mathfrak{o}_A
\end{equation}
in which both terms are representations of $A$: $\mathcal{D}_{A}^{\textrm{tr}}$ was already discussed, while $\mathfrak{o}_A$ is a representation of $A$ since it is a flat vector bundle over $M$. 

\begin{lemma}\label{mod-cor0}   The representations $Q_A$, $\mathcal{D}_{A}^{\mathrm{tr}}$ and $\mathcal{V}_{A}^{\mathrm{tr}}$ of $A$ have the same characteristic class (namely $\mathrm{mod}(A)$).
\end{lemma}

\begin{proof} Using $\mathcal{D}_{A}^{\textrm{tr}}\cong \mathcal{V}_{A}^{\textrm{tr}}\otimes \mathfrak{o}_{A}^{\textrm{tr}}$, (\ref{dec-QA}) and the multiplicativity of $\theta_{A}$, it suffices to show that $\theta_A(\mathfrak{o}_{A}^{\textrm{tr}})= 0$ and similarly for $\mathfrak{o}_{A}$. Using again multiplicativity, it suffices to show that $\theta_A(\mathfrak{o}_{A}^{\textrm{tr}}\otimes \mathfrak{o}_{A}^{\textrm{tr}})= 0$ - which is true because $\mathfrak{o}_{A}^{\textrm{tr}}\otimes \mathfrak{o}_{A}^{\textrm{tr}}$ is isomorphic to the trivial representation (Lemma \ref{lemma-can-izos}); and similarly for $\mathfrak{o}_{A}$ just that, this time, $\mathfrak{o}_{A}\otimes \mathfrak{o}_{A}$ is isomorphic to the trivial line bundle already as a flat vector bundle.
\end{proof}

Despite the previous lemma, because of Proposition \ref{not-col-tr} and the discussion around it, using $\mathcal{V}_{A}^{\mathrm{tr}}$ to define $\textrm{mod}(A)$, although correct, may give rise to the wrong expectations. However, using $Q_A$ to define $\textrm{mod}(A)$ is even more unfortunate, for even more fundamental reasons: in general, $Q_A$ is not a representation of the groupoid $\G$! Indeed, using (\ref{dec-QA}), the condition that $Q_A$ can be made into a representation of $\G$ is equivalent to the same condition for $\mathfrak{o}_A$. But the orientation bundles 
$\mathfrak{o}_A$ are the typical examples of algebroid representations that do not come from groupoid ones. That is clear already in the case of the pair groupoid of a manifold $M$, whose representations are automatically trivial as vector bundles, but for which $\mathfrak{o}_{A}= \mathfrak{o}_{TM}$ is not trivializable if $M$ is not orientable. 

Note also that the fact that $\mathcal{D}_{A}^{\textrm{tr}}$, unlike $Q_A$, is a representation of $\G$, was absolutely essential for obtaining Proposition \ref{mod-cor1} and Corollary \ref{mod-cor2}.

\subsection{Transverse orientability and the first Stiefel-Whitney class of $\G$} 
It is interesting to look back at the construction of the characteristic class $\theta_{\G}(L)$ of a $1$-dimensional representation $L$ of $\G$. The reason for the warning mentioned in Remark \ref{remark-warning} comes from the fact that, when passing from $\tilde{c}_{\sigma}$ to $c_{\sigma}$ in (\ref{tilde-c-sigma-form}), one loses information related to orientability. 

Following the exposition in \cite{tese}, we return to the discussion around  (\ref{tilde-c-sigma-form}); in particular, we assume that $L$ is a $1$-dimensional representation of $\G$ that is trivializable as a vector bundle, and $\sigma$ is a nowhere vanishing section of $L$. Then one can either:
\newpage
\begin{itemize}
\item consider the entire $\tilde{c}_{\sigma}: \G \rightarrow \mathbb{R}^*$
\item consider only the part of $\tilde{c}_{\sigma}$ that is not contained in $c_{\sigma}$, i.e., 
\[ \textrm{sign} \circ \tilde{c}_{\sigma}: \G \rightarrow \mathbb{Z}_2,\]
where we identify $\mathbb{Z}_2$ with $\{-1, 1\}\subset \mathbb{R}^*$.
\end{itemize}
Both of them are (differentiable) cocyles on $\G$ with coefficients in a (abelian Lie) group. Such cocycles give rise to classes in the cohomology groups 
$H^{1}_{\textrm{diff}}(\G, \mathbb{R}^*)$ and $H^{1}_{\textrm{diff}}(\G, \mathbb{Z}_2)$, which are abelian groups but no longer vector spaces. As before, the resulting cohomology classes are independent of $\sigma$; we denote them by
\[ \tilde{\theta}_{\G}(L)\in H^{1}_{\textrm{diff}}(\G, \mathbb{R}^*), \ \ w(L)\in H^{1}_{\textrm{diff}}(\G, \mathbb{Z}_2),\]
and we will call them the \textbf{extended characteristic class} of $L$, and the \textbf{Stiefel-Whitney class of $L$}, respectively. All these classes can be put together using the decomposition of the group $\mathbb{R}^*$ as
\[ \mathbb{R}^* \cong \mathbb{R} \times \mathbb{Z}_2,\ \ \lambda \mapsto (\textrm{ln}(|\lambda|), \textrm{sign}(\lambda))\]
and the induced isomorphism 
\[ H^{1}_{\textrm{diff}}(\G, \mathbb{R}^*) \cong H^{1}_{\textrm{diff}}(\G) \times H^{1}_{\textrm{diff}}(\G, \mathbb{Z}_2) .\]
With this, the extended class of $L$ is 
\[ \tilde{\theta}_{\G}(L)= (\theta_{\G}(L), w(L)). \]
Note that, by construction, $\tilde{\theta}_{\G}(L)$ is trivial (equal to the identity of the group) if and only if $L$ is isomorphic to the trivial representation, and $w(L)$ is trivial if and only if $L$ is $\G$-orientable, i.e., if $L$ admits an orientation with the property that the action of $\G$ is orientation-preserving. A warning however: the classes $\tilde{\theta}_{\G}(L)$ and $w(L)$ have been defined so far only when $L$ is trivializable as a vector bundle; moreover, these constructions cannot be extended to general $L$'s while preserving their main properties. 
Hence, when it comes to the canonical representations of $\G$, one can apply them only to $\mathcal{D}_{A}^{\textrm{tr}}$, for which one obtains
\[ w(\mathcal{D}_{A}^{\textrm{tr}})= 1, \ \ \tilde{\theta}_{\G}(\mathcal{D}_{A}^{\textrm{tr}})= (\textrm{mod}(\G), 1).\]
For the previous discussion we assumed that $L$ was trivializable as a vector bundle. 
To handle general $L$'s one can use covers $\mathcal{U}$ of $M$ by open subsets over which $L$ is trivializable. Such an open cover induces a groupoid $\G_{\mathcal{U}}$ over the disjoint union of the open subsets in $\mathcal{U}$, obtained by pulling-back $\G$ along the canonical map from the disjoint union into $M$. The pull-back $L_{\mathcal{U}}$ of $L$ is a representation of 
$\G_{\mathcal{U}}$ and, by the choice of $\mathcal{U}$, one has well-defined classes
\[ \tilde{\theta}(L_{\mathcal{U}})= (\theta(L_{\mathcal{U}}), w(L_{\mathcal{U}}))\in H^{1}_{\textrm{diff}}(\G_{\mathcal{U}}, \mathbb{R}^*)\cong H^{1}_{\textrm{diff}}(\G_{\mathcal{U}})\times H^{1}_{\textrm{diff}}(\G_{\mathcal{U}}, \mathbb{Z}_2).\]
Note that, since $\G_{\mathcal{U}}$ is Morita equivalent to $\G$ (see Example $5.10$ in \cite{ieke_mrcun}), when passing from $L$ to $L_{\mathcal{U}}$ (as representations) one does not lose any information. To obtain a class that is independent of the covers one proceeds as usual and one passes to the filtered colimit (with respect to the refinement of covers) and defines
\[ \check{\mathrm{H}}^{1}_{\textrm{diff}}(\G, \mathbb{R}^*)= \lim_{\to \mathcal{U}} H^{1}_{\textrm{diff}}(\G_{\mathcal{U}}, \mathbb{R}^*),\]
and similarly $\check{\mathrm{H}}^{1}_{\textrm{diff}}(\G, \mathbb{Z}_2)$. The Morita invariance of differentiable cohomology with coefficients in $\mathbb{R}$ implies that the restriction to open subsets induces an isomorphism
\[  H^{1}_{\textrm{diff}}(\G)\cong H^{1}_{\textrm{diff}}(\G_{\mathcal{U}})\]
(that sends $\theta_{\G}(L)$ to $\theta_{\G_{\mathcal{U}}}(L_{\mathcal{U}})$); hence there are induced cohomology classes
\[ \check{\theta}_{\G}(L)\in \check{\mathrm{H}}^{1}_{\textrm{diff}}(\G, \mathbb{R}^*), \ \ w(L)\in \check{\mathrm{H}}^{1}_{\textrm{diff}}(\G, \mathbb{Z}_2)\]
and canonical isomorphism of groups 
\[  \check{\mathrm{H}}^{1}_{\textrm{diff}}(\G, \mathbb{R}^*) \cong H^{1}_{\textrm{diff}}(\G)\times \check{\mathrm{H}}^{1}_{\textrm{diff}}(\G, \mathbb{Z}_2)\]
such that:
\begin{itemize}
\item $\check{\theta}_{\G}(L)= (\theta_{\G}(L), w(L))$.
\item $\check{\theta}_{\G}(L)$ is trivial if and only if $L$ is isomorphic to the trivial representation.
\end{itemize}
Actually,  $\check{\theta}_{\G}$ gives an isomorphism between the group $\textrm{Rep}^1(\G)$ of isomorphism classes of $1$-dimensional representations (with the tensor product) with the \v{C}ech-type cohomology with coefficients in $\mathbb{R}^*$:
\begin{equation}\label{tautol-class} 
\check{\theta}_{\G}: \textrm{Rep}^1(\G) \stackrel{\sim}{\rightarrow} \check{\mathrm{H}}^{1}_{\textrm{diff}}(\G, \mathbb{R}^*).
\end{equation}
Denoting $\mathbb{R}^*= GL_1(\mathbb{R})$ by $H$, this is a particular case of the interpretation of $\G$-equivariant principal $H$-bundles in terms of transition functions (see for example \cite{husemoller}), valid for any Lie group $H$, interpretation that is itself at the heart of Haefliger's work on the transverse geometry of foliations \cite{haefliger}. While $w$ is trivial on $\mathcal{D}_{A}^{\textrm{tr}}$, it gives rise to interesting information when applied to the transverse volume bundle or, equivalently, to the orientation one. 

\begin{definition} The \textbf{transverse first Stiefel-Whitney class} of $\G$ is: 
\[ w^{\textrm{tr}}_{1}(\G):= w(\mathcal{V}_{A}^{\textrm{tr}})= w(\mathfrak{o}_{A}^{\textrm{tr}})\in \check{\mathrm{H}}^{1}_{\textrm{diff}}(\G, \mathbb{Z}_2).\]
\end{definition}

As a consequence of the previous discussion we state here the following:

\begin{corollary} One has:
\begin{enumerate}
\item[1.] $\G$ is transversely orientable iff $w^{\mathrm{tr}}_{1}(\G)= 1$.
\item[2.] $\G$ admits transverse volume forms iff $\mathrm{mod}(\G)= 0$ and $w^{\mathrm{tr}}_{1}(\G)= 1$.
\end{enumerate}
\end{corollary}

\noindent \textbf{Acknowledgements}

This research was supported by the NWO Vici Grant no. 639.033.312. The second author was supported also by the FCT grant  SFRH/BD/71257/2010 under the POPH/FSE programmes. We would also like to acknowledge various discussions with Rui Loja Fernandes, Ioan M\unichar{259}rcu\unichar{539} and David Mart\'inez Torres.

\bibliography{mybiblio}
\bibliographystyle{amsplain}

\end{document}